\documentclass[10pt]{amsart}
\usepackage{amssymb}
\usepackage{dsfont}
\usepackage{amscd}
\usepackage[mathscr]{euscript} 
\usepackage[all]{xy}
\usepackage{url}
\usepackage{stmaryrd}
\usepackage{comment}
\usepackage[retainorgcmds]{IEEEtrantools}
\usepackage{hyperref}
\usepackage{graphicx}
\title{Strongly minimal reducts of valued fields}
\author[P. KOWALSKI]{Piotr Kowalski$^{\spadesuit}$}
\thanks{$^{\spadesuit}$Supported by T\"{u}bitak grant 2221 and NCN grant 2012/07/B/ST1/03513}
\address{$^{\spadesuit}$Instytut Matematyczny\\
Uniwersytet Wroc{\l}awski\\
Wroc{\l}aw\\
Poland}
\email{pkowa@math.uni.wroc.pl} \urladdr{http://www.math.uni.wroc.pl/\textasciitilde pkowa/ }
\author[S. RANDRIAMBOLOLONA]{Serge Randriambololona$^{\dagger}$}
\address{$^{\dagger}$Galatasaray \"{U}n\.{\.i}vers\.{\.i}tes\.{\.i}\\
Istanbul\\
Turkey}
\email{serge.randriambololona@math.cnrs.fr}

\DeclareMathOperator{\acl}{acl}  
\DeclareMathOperator{\gl}{GL} \DeclareMathOperator{\aut}{Aut} 
\DeclareMathOperator{\stab}{Stab}

\DeclareMathOperator{\ch}{char}

\DeclareMathOperator{\ddf}{DF}\DeclareMathOperator{\dcf}{DCF}
\DeclareMathOperator{\acvf}{ACVF}\DeclareMathOperator{\acf}{ACF}\DeclareMathOperator{\rmm}{RM}\DeclareMathOperator{\dmm}{DM}

\newtheorem{theorem}{Theorem}[section]
\newtheorem{prop}[theorem]{Proposition}
\newtheorem{lemma}[theorem]{Lemma}

\newtheorem{conjecture}[theorem]{Conjecture}

\theoremstyle{definition}
\newtheorem{definition}[theorem]{Definition}

\newtheorem{remark}[theorem]{Remark}

\newtheorem{assumption}[theorem]{Assumption}

\begin{document}

\newcommand{\twoc}[3]{ {#1} \choose {{#2}|{#3}}}
\newcommand{\thrc}[4]{ {#1} \choose {{#2}|{#3}|{#4}}}
\newcommand{\Zz}{{\mathds{Z}}}
\newcommand{\Ff}{{\mathds{F}}}
\newcommand{\Cc}{{\mathds{C}}}
\newcommand{\Rr}{{\mathds{R}}}
\newcommand{\Nn}{{\mathds{N}}}
\newcommand{\Qq}{{\mathds{Q}}}
\newcommand{\Kk}{{\mathds{K}}}
\newcommand{\Pp}{{\mathds{P}}}
\newcommand{\ddd}{\mathrm{d}}
\newcommand{\Aa}{\mathds{A}}
\newcommand{\dlog}{\mathrm{ld}}
\newcommand{\ga}{\mathbb{G}_{\rm{a}}}
\newcommand{\gm}{\mathbb{G}_{\rm{m}}}
\newcommand{\gaf}{\widehat{\mathbb{G}}_{\rm{a}}}
\newcommand{\gmf}{\widehat{\mathbb{G}}_{\rm{m}}}
\newcommand{\gdf}{\mathfrak{g}-\ddf}
\newcommand{\gdcf}{\mathfrak{g}-\dcf}
\newcommand{\fdf}{F-\ddf}
\newcommand{\fdcf}{F-\dcf}

\maketitle
\begin{abstract}
We prove that if a strongly minimal non-locally modular reduct of an algebraically closed valued field of characteristic $0$ contains $+$, then this reduct is bi-interpretable with the underlying field.
\end{abstract}

\section{Introduction}\label{secintro}

In 1980's, Zilber posed a conjecture \cite{Zil1} asserting that if a strongly minimal structure is not locally modular, then it interprets a field. Zilber's conjecture was refuted by Hrushovski \cite{HR2}, however it holds for many interesting classes of structures. General feeling is that Zilber's conjecture should hold in a ``geometric context''. This feeling is confirmed by a theorem of Hrushovski and Zilber \cite{HZ} which says that Zilber's conjecture holds for strongly minimal \emph{Zariski geometries}.

During the problem session of the Pure Model Theory conference in Norwich (July 2005), Kobi Peterzil asked whether Zilber's conjecture holds for strongly minimal structures \emph{interpretable in o-minimal ones}. We will refer to this (still open) question as \emph{Peterzil's conjecture}. In this case it seems difficult to put a Zariski geometry structure on the strongly minimal structure, since the ambient o-minimal geometry is far from being Zariski.

In this paper, we consider a \emph{valued field version} of Peterzil's conjecture. We formulate below its direct translation to the valued field context. By $\acvf$ (resp. $\acvf_0$) we mean the theory of algebraically closed non-trivially valued fields considered in the language of rings with an extra unary relation symbol for the valuation ring (resp. of characteristic $0$, with no restrictions on the residue characteristic).
\begin{conjecture}[Valued field version of Peterzil's Conjecture]\label{conjecture} Let $M$ be a strongly minimal structure which is not locally modular and interpretable in an algebraically closed valued field. Then $M$ interprets a field.
\end{conjecture}
\begin{remark}
A more general version of Conjecture \ref{conjecture} can be obtained by replacing ``algebraically closed valued field'' with ``$C$-minimal field'' or even  ``$C$-minimal structure'' (see \cite{HM1}). We will discuss such possible generalizations in Section \ref{seccmin}.
\end{remark}
If $\mathbf{M}$ is an o-minimal structure or a model of ACVF, then the $\acl_{\mathbf{M}}$-operator is a pregeometry giving $\mathbf{M}$ a notion of dimension $\dim_{\mathbf{M}}$ on tuples and definable sets. Let $\mathcal{M}$ be a strongly minimal structure interpretable in $\mathbf{M}$. In the case when the universe of $\mathcal{M}$ is definable in $\mathbf{M}$, it is natural to start attacking Peterzil's conjecture from the cases of $\mathcal{M}$ of small $\dim_{\mathbf{M}}(\mathcal{M})$.

In \cite{HaOPe}, Peterzil's conjecture was verified in the case of $\dim_{\mathbf{M}}(\mathcal{M})=1$ (showing that $\mathcal{M}$ is then locally modular). In an attempt to attack the case of $\dim_{\mathbf{M}}(\mathcal{M})=2$, Assaf Hasson and the first author (motivated by \cite{MaPi}) considered in \cite{HaKo} the case where the structure $\mathcal{M}$ expands $(\Cc,+)$ (or, more generally, the additive group of the algebraic closure of the underlying o-minimal field). By a general model-theoretic argument (see Proposition \ref{reduction}), in such a case Peterzil's conjecture reduces to the case of a strongly minimal structure of the form $(\Cc,+,X)$, where $X$ is an $\mathbf{M}$-definable subset of $\Cc\times \Cc$. It is shown in \cite{HaKo} (after quite a long argument) that if $X$ is a graph of a function $f$, then $f$ is an $\Rr$-linear conjugate of a $\Cc$-constructible function, which, in particular, verifies Peterzil's conjecture in this case.

In the valued field context, the situation seems to be easier than in the o-minimal one. The 1-dimensional valued field  case corresponds to the 2-dimensional o-minimal case (since the underlying valued field is already algebraically closed). In this paper we show the following.
\begin{theorem}\label{acvfthm}
Suppose $\mathbf{K}=(K,+,\cdot,\mathcal{O}_K)$ is an algebraically closed valued field of characteristic $0$ and $\mathcal{K}$ is a strongly minimal reduct of $\mathbf{K}$ containing $(K,+)$. If $\mathcal{K}$ is not locally modular, then $\mathcal{K}$ is bi-interpretable with the field $(K,+,\cdot)$.
\end{theorem}
\begin{remark}
\begin{enumerate}
\item Theorem \ref{acvfthm} extends \cite[Thm 2.1]{MaPi} from the case of $\acf_0$ to the case of $\acvf_0$.

\item The proof of \cite[Thm 2.1]{MaPi} goes through to give an Archimedean version of Theorem \ref{acvfthm} (see also \cite[Remark 7.5(2)]{HaKo}).
\end{enumerate}
\end{remark}

This paper is organized as follows. In Section 2, we collect classical results about valued fields which we need.  In Section 3, we present a proof of Theorem \ref{acvfthm}. In Section 4, we discuss possible generalizations of Theorem \ref{acvfthm} beyond the context of pure algebraically closed valued fields.

We would like to thank Assaf Hasson for his comments on an earlier version of this paper. We would also like to thank to the referee for her/his very careful reading of the paper and many valuable comments.

\section{Preliminaries on valued fields}\label{secnec}

In this section, we collect the classical results from non-Archimedean analysis and model theory of valued fields which will be used in Section \ref{secproof}. For reader's convenience, we will state these results in the simplest possible form (so the lowest generality) which we will need in the sequel. We fix $\mathbf{K}=(K,+,\cdot,\mathcal{O}_K)$, an algebraically closed valued field. We denote the corresponding (multiplicative) valuation by $|\cdot |$.

To prevent a possible confusion we would like to point out the following.
\begin{remark}\label{twoar}
There are two notions of ``Archimedean''.
\begin{enumerate}
\item An Archimedean (normed) field, e.g. a normed subfield of the field of complex numbers.
\item An Archimedean (ordered) group, e.g. an ordered subgroup of $(\Rr,+)$.
\end{enumerate}
In this paper we are concerned with non-Archimedean fields, but sometimes we discuss their connections with Archimedean fields. We will also need to work with valued fields having the Archimedean value group (see Section \ref{secstan})
\end{remark}
We need two $\mathbf{K}$-analytic results. For readers convenience we recall the necessary definitions (taken from \cite{abhyankar1964local}) below. We assume that here (until Theorem \ref{qe}) the \emph{value group} of $\mathbf{K}$ is Archimedean (see Remark \ref{twoar}).
\begin{definition}
We assume that $\mathbf K$ is complete. Suppose that $U$ is an open subset of $K^n$ for some $n$, $\phi$ is a map from $U$ to $K$ and $a\in U$.
\begin{itemize}
\item The map $\phi$ is said to be \emph{$\mathbf K$-analytic at $a$} if there is a formal power series $f$ in $n$ variables with coefficients in $K$, converging in an open neighbourhood $\Omega$ of $0$, such that $a+b\in  U$ and $\phi(a+b)=f(b)$ for each $b\in \Omega$.

\item The map $\phi$ is said to be \emph{$\mathbf K$-analytic on $U$} if it is $\mathbf K$-analytic at every point of $U$.

\item A set $X\subset K^n$ is said to be \emph{$\mathbf K$-analytic at $a$} if there is a neighbourhood $V$ of $a$ in $K^n$ and finitely many $\mathbf K$-analytic functions $\phi_1$, \dots , $\phi_l$ on $V$ such that
\[
X\cap V =\{ x\in V\ |\ \phi_1(x)=0,\ldots ,\phi_l(x)=0\}.
\]

\item A set $X\subset K^n$ is said to be \emph{locally $\mathbf K$-analytic} if it is $\mathbf K$-analytic at each of its points.
\end{itemize}
\end{definition}

\begin{theorem}[Implicit Function Theorem, page 84 of \cite{abhyankar1964local}]\label{ift}
We assume that $\mathbf{K}$ is complete. Suppose $U$ is an open subset of $K\times K$, $F:U\to K$ is a $\mathbf{K}$-analytic function and $z\in U$. If $\frac{\partial F}{\partial y}(z)\neq 0$, then there are open sets $U_x,U_y\subseteq K$ such that $z\in U_x\times U_y\subseteq U$ and there is a $\mathbf{K}$-analytic function $f:U_x\to U_y$ such that we have the following
$$\{v\in U_x\times U_y\ |\ F(v)=0\} = \{(a,f(a))\ |\ a\in U_x\}.$$
\end{theorem}

\begin{theorem}[Continuity of Roots]\label{ap}
We assume that $\mathbf{K}$ is complete. Let $U,U'\subseteq K$ be open subsets and $F:U\times U'\to K$ be a $\mathbf{K}$-analytic function. For any $t\in U'$, we denote the function $F(\cdot,t)$ by $f_t(\cdot)$. Suppose that there are elements $a\in U,b\in U'$ such that the function $f_b$ has a zero of multiplicity $d>0$ at $a$. Then there are open subsets $a\in U_a\subseteq U,\ b\in U_b\subseteq U'$ such that:
\begin{itemize}
\item $f_b^{-1}(0)\cap U_a=\{a\}$,

\item for any $t\in U_b$, the function $f_t$ has exactly $d$ zeroes (counting multiplicities) in $U_a$.
\end{itemize}
\end{theorem}
\begin{proof}
It can be proved as in the Archimedean case. One can either use Weierstrass Preparation Theorem (more precisely: \cite[(10.3)2)]{abhyankar1964local} and \cite[(11.3)]{abhyankar1964local}) or Argument Principle (discussed in the non-Archimedean case e.g. in \cite{Maa}).
\end{proof}

We also need an algebraic result about valued fields.
\begin{theorem}[K\"{u}rsch\'{a}k, Theorem on p. 142 of \cite{ribenboim1999theory}]\label{compacf}
The completion of an algebraically closed valued field is again algebraically closed.
\end{theorem}

Finally, we need some model theory of algebraically closed valued fields (with an arbitrary value group). The first result is due to Robinson and the second to Holly.
\begin{theorem}[Model Completeness, Section 3.5 in \cite{robinson1956complete}]\label{qe}
Any extension of non-trivially valued algebraically closed valued fields is elementary.
\end{theorem}

\begin{theorem}[Swiss Cheese Decomposition, \cite{holly1995}]\label{swiss}
A $\mathbf{K}$-definable subset of $K$ is a union of \emph{Swiss cheeses}. In particular, an infinite $\mathbf{K}$-definable subset of $K$ has non-empty interior.
\end{theorem}

\section{The proof}\label{secproof}

In this section we prove Theorem \ref{acvfthm}. We assume that $\ch(K)=0$. Let $\mathcal{K}$ be a strongly minimal reduct of $\mathbf{K}$ containing $(K,+)$. We assume that $\mathcal{K}$ is not locally modular.

\subsection{Standard reductions}\label{secstan}
In this subsection we will show that we can simplify both of the structures $\mathbf{K}$ and $\mathcal{K}$ without loss of generality. We note first quite an obvious fact.
\begin{lemma}\label{ee}
Suppose $\mathbf{L}$ is an algebraically closed valued field which is elementary equivalent to $\mathbf{K}$. Then we have the following.
\begin{enumerate} 
\item The corresponding strongly minimal reduct $\mathcal{L}$ (of $\mathbf{L}$) contains $(L,+)$ and is not locally modular.

\item If $\mathcal{L}$ is bi-interpretable with the field $(L,+,\cdot)$, then  $\mathcal{K}$ is bi-interpretable with the field $(K,+,\cdot)$.
\end{enumerate}
\end{lemma}
\begin{proof}
For item $(1)$, it is enough to notice that local modularity is preserved under elementary equivalence. The item $(2)$ is clear.
\end{proof}

By Theorems \ref{compacf} -- \ref{qe} and Lemma \ref{ee}, we can assume that $\mathbf{K}$ is complete with the Archimedean value group, so Theorems \ref{ift} -- \ref{ap} can be applied. For the same reason, we can also assume that $K$ is uncountable, which we will need as a very mild form of saturation.

The following result is ``folklore''. We assume that the language is countable.
\begin{prop}\label{reduction}
Assume that $\mathcal{A}=(A,+,\ldots)$ is an uncountable strongly minimal group which is not locally modular. Then there is an $\mathcal{A}$-definable $X\subseteq A\times A$ such that the structure $(A,+,X)$ is not locally modular.
\end{prop}
\begin{proof}
Take a two-dimensional definable family $(X_c)_{c\in C}$ of strongly minimal subsets of $A\times A$ which exists by Prop. 2.6 in Section 2 of \cite{Pi} (no saturation is needed here). Since $A$ is uncountable, there is $c\in C$ such that $\dim_{\mathcal{A}}(c)=2$. Let us define $X$ as $X_c$. If the structure $(A,+,X)$ is locally modular, then by Corollary 4.8 in Section 4 of \cite{Pi}, we have $X=a+H$ where $a\in A\times A$ and $H$ is an $\acl(\emptyset)$-definable (in the structure $(A,+,X)$, hence also in the structure $\mathcal{A}$) subgroup of $A\times A$. But then the family $(X_c)_{c\in C}$ is at most one-dimensional, a contradiction.
\end{proof}
Therefore, without loss of generality, we may assume the following.
\begin{assumption}\label{firstass}
The structure $\mathcal{K}$ coincides with $(K,+,X)$, where $X\subseteq K^2$ is a $\mathbf{K}$-definable set, and $K$ is uncountable.
\end{assumption}

\subsection{Decompositions}\label{secdec}
In this subsection, we fix $Y\subseteq K\times K$ which is $\mathbf{K}$-definable and infinite. The first lemma is well-known.
\begin{lemma}\label{decomp1}
We can write $Y$ as a finite union $Y=Y_1\cup \ldots \cup Y_r$ where each $Y_i$ is an open subset (in the valuation topology) of a Zariski closed subset of $\Aa^2(K)$.
\end{lemma}
\begin{proof}
By quantifier elimination for $\mathbf{K}$ (see e.g. \cite[Thm. 7.1(ii)]{HHM2}), the set $Y$ is of the form
$$\bigcup_i\bigcap_j\{(a,b)\in K\times K\ :\ |F_{ij}(a,b)|\boxtimes_{ij} |G_{ij}(a,b)|\},$$
where $F_{ij},G_{ij}\in K[x,y]$ and $\boxtimes_{ij}\in \{=,<\}$. Without loss of generality, we can skip the $\bigcup_i$-symbol.
If $\boxtimes_{ij}$ is $<$, then the corresponding definable set is open in the valuation topology. If $\boxtimes_{ij}$ is $=$ and $F_{ij}G_{ij}\neq 0$, then the corresponding definable set is still open in the valuation topology. Finally, if $\boxtimes_{ij}$ is $=$ and $F_{ij}G_{ij}=0$, then the corresponding definable set is Zariski closed, so the lemma follows.
\end{proof}
\begin{assumption}\label{assume}
Let us assume that $Y$ is a \emph{boundary set}, i.e. it has empty interior in the valuation topology. By Lemma \ref{decomp1}, $Y$ is boundary if and only if the Zariski closure of $Y$ is a proper subset of $\Aa^2(K)$. We write now $Y$ as a finite disjoint union $Y=Y_0\cup \ldots \cup Y_r$, where $Y_0$ is finite and for each $i>0$, the Zariski closure of $Y_i$, denoted by $V_i$, is an irreducible algebraic curve, i.e. $V_i$ is the set of zeroes of an irreducible polynomial $F_i$. Finally, we assume that for $1\leqslant i<j\leqslant r$, we have $V_i\neq V_j$. Such a presentation of $Y$ is unique up to the choice of the finite set $Y_0$.
\end{assumption}
From now on, we ``privilege'' the first coordinate over the second one, i.e. we call a subset $Z\subseteq K\times K$ a \emph{graph of a function} if there is a (partial) function $f:K\to K$ such that
$$Z=\{(a,f(a))\ |\ a\in K\}.$$
We define the following ``bad locus'' set related to $Y$
$$Z_Y:=\left\{a\in Y\ |\ \bigvee_{i>0}\ \frac{\partial F_i}{\partial y}(a)=0\right\}\cup Y_0.$$
\begin{remark}
The bad locus set $Z_Y$ depends also on the chosen presentation of $Y$ as $Y=Y_0\cup \ldots Y_r$. By Assumption \ref{assume} (last sentence), $Z_Y$ depends only on $Y$ and the choice of $Y_0$. We prefer to write $Z_Y$ instead of $Z_{Y,Y_0}$.
\end{remark}
\begin{lemma}\label{anfunc}
For all $a\in Y\setminus Z_Y$, there is an open (in the valuation topology) subset $U\ni a$ contained in $Y\setminus Z_Y$ such that $Y\cap U$ is the graph of an analytic function.
\end{lemma}
\begin{proof}
By Lemma \ref{decomp1}, $Y$ is locally (in the valuation topology) Zariski closed, hence locally analytic.
Therefore Theorem \ref{ift} gives the result.
\end{proof}
Let $f$ denote the function given by Lemma \ref{anfunc}. For a given $a=(a_1,a_2)\in Y\setminus Z_Y$, we define the following:
$$Y'(a):=f'(a_1).$$
\begin{lemma}\label{deffunc}
The function
$$Y\setminus Z_Y\ni a\mapsto Y'(a)\in K$$
is definable in the structure $\mathbf{K}$.
\end{lemma}
\begin{proof}
If $f:K\to K$ is a function definable in $\mathbf K$, the derivative of $f$ is definable in $\bf K$ (uniformly in the parameters used to define $f$) by the usual $\varepsilon \delta$-formula:
\[
\ell=f'(x)\ \ \ \Leftrightarrow\ \ \ (\forall \varepsilon\neq 0) (\exists \delta\neq 0) (\forall t) (t\delta^{-1} \in \mathcal{O}_K \rightarrow f(x+t)-f(x)-\ell t \in t\varepsilon\mathcal{O}_K ).
\]
Consider the following set
\[
\left\{(x,y,m)\in K^3\ |\ (x,y)\in Y\setminus Y_0 \ \wedge\ \bigvee_i \left(F_i(x,y)=0\ \wedge \ \frac{\partial F_i}{\partial t} (x+t,y+mt)|_{t=0}=0\right) \right\}.
\]
This set is precisely the graph of the function $a\mapsto Y'(a)$.
\end{proof}

\subsection{Calculus of derivatives}\label{seccalc}
In this subsection, we assume that $Y\subseteq K\times K$ is $\mathcal{K}$-definable and strongly minimal (i.e. $\rmm_{\mathcal{K}}(Y)=1$ and $\dmm_{\mathcal{K}}(Y)=1$). Assumption \ref{assume} is satisfied, since if $Y$ is not a boundary set, then $\rmm(Y)=2$.
\begin{lemma}\label{needed}
Suppose there is $i$ such that $V_i$ is an affine line. Then we have the following.
\begin{enumerate}
\item The set $Y$ is an open subset of an affine line up to a finite set.

\item The structure $(K,+,Y)$ is locally modular.
\end{enumerate}
\end{lemma}
\begin{proof}
By Assumption \ref{assume}, the set $Y$ can be written as $Y=Y_0\cup \ldots \cup Y_r$ where $Y_0$ is finite and for $i>0$, the set $Y_i$ is open (in the valuation topology) inside the irreducible algebraic curve $V_i$. We can assume that $Y_0=\emptyset$ and that $V_1$ is an affine line. Without loss of generality, we can also assume that $(0,0)\in Y_1$, hence $V_1$ is a linear subspace of $K^2$. Since open subgroups of $K^2$ form a neighborhood basis for a point $(0,0)$, there is a subgroup $B$ such that
$$B\cap Y=B\cap V_1.$$
By our assumptions, $B\cap V_1$ is a subgroup of $K^2$. Then $B\cap V_1$ is a subgroup of the \emph{stabilizer} of (the $\mathcal{K}$-generic type of) $Y$, which we denote by $\stab(Y)$. (For the notion of a stabilizer which is used in this proof and its properties, the reader is referred to \cite[Sec 1.6]{Pi}.)
In particular, $\stab(Y)$ is infinite. Since $Y$ is strongly minimal, it coincides up to a finite set with a coset of $\stab(Y)$. Without loss of generality, $Y$ is a coset of $\stab(Y)$. Since $(0,0)\in Y$, we get that $Y$($=\stab(Y)$) is a subgroup of $K^2$.

For the proof of $(1)$, it is enough to show that $r=1$. Take $i\in \{1,\ldots,r\}$. Let $a\in Y_i$ and $B'\ni a$ be a coset of an open subgroup such that
$$Y\cap B'=V_i\cap B'.$$
Since $Y$ is a subgroup of $K^2$ and $B'$ is a coset of a subgroup of $K^2$, we get
$$Y\cap B'=a+Y',$$
where $Y'$ is an open subgroup of $Y$. Since $Y'$ and $B\cap Y$ are both open subgroups of $Y$, the set $Y'':=Y'\cap B$ is an open subgroup of $Y$. Because  $Y''$ is an infinite subset of $V_1$, we get that $a+Y''$ is an infinite subset of $V_i$. Then two irreducible curves $a+V_1$ and $V_i$ coincide, because they have infinite intersection. Therefore, the image of $Y$ in the quotient vector space $K^2/V_1$ has $r$ elements. But this quotient is also a subgroup of a torsion free (as a vector space over a field of characteristic $0$) group, hence $r=1$.

For the proof of $(2)$, note that the projection on one of the coordinate axis is a definable one-to-one map on $Y$. By strong minimality of $\mathcal{K}$, the image of this projection is cofinite. Hence the structure $(K,+,Y)$ is definable in the locally modular structure $(K,+,\cdot_{\lambda})_{\lambda\in K}$, therefore it is locally modular itself (see \cite{Wa3}).
\end{proof}
\begin{remark}\label{comingfromx}
We will apply Lemmas \ref{decomp1} -- \ref{needed} for $Y$ being $X$ or a curve ``coming from $X$'', i.e. a curve obtained by applying the operations $+,-,\circ$ described before Lemma \ref{leibnitz} to the additive translates of $X$.
\end{remark}

\begin{lemma}\label{zxfinite}
The set $Z_X$ is finite.
\end{lemma}
\begin{proof}
If $Z_X$ is infinite, then some $F_i$ (for $Y=X)$ is linear and by Lemma \ref{needed}, $\mathcal{K}$ is locally modular, a contradiction.
\end{proof}
By Lemma \ref{zxfinite}, without loss of generality, we can assume that $Z_X=\emptyset$. We also assume that $(0,0)\in X$.
\\
\\
As in \cite[Def. 2.9]{MaPi}, for any $V,W\subseteq K\times K$ we define the following:
\begin{itemize}
\item $V+W:=\{(x,y_1+y_2)\ |\ (x,y_1)\in V,(x,y_2)\in W\}$;

\item $-V:=\{(x,-y)\ |\ (x,y)\in V\}$;

\item $V\circ W=\{(t_1,t_2)\in K\times K\ |\ (\exists t\in K)((t_1,t)\in W\ \wedge\ (t,t_2)\in V)\}$;

\item for $a\in V$, $V_a:=\{z-a\ |\ z\in V\}$.
\end{itemize}
As in \cite[2.10]{MaPi}, we obtain the following lemma.
\begin{lemma}\label{leibnitz}
Let $V,W\subseteq K\times K$ be $\mathbf{K}$-definable. Then we have the following.
\begin{enumerate}
\item If $V$ an $W$ satisfy Assumption \ref{assume}, then $V+W$ and $V\circ W$ satisfy Assumption \ref{assume} as well. Moreover, if $Z_V=\emptyset=Z_W$, then $Z_{V+W}=\emptyset=Z_{V\circ W}$.


\item For any $a,b\in V\setminus Z_V$ we have
$$(V_a+V_b)'(0)=V'(a)+V'(b),\ \ \ (V_a\circ V_b)'(0)=V'(a)\cdot V'(b).$$
\end{enumerate}
\end{lemma}

\subsection{Field configuration}\label{secfieldcon}
For the notion of \emph{group configuration} and how a type-definable transitive action of a type-definable group on a type-definable set gives a group configuration, the reader is advised to consult \cite[Section 5.4]{Pi}. The following theorem of Hrushovski may be obtained as a combination of \cite[Theorem 5.4.5]{Pi} and the theorem saying that any type-definable group in an $\omega$-stable theory is definable \cite[Corollary 5.19]{Po1}.
\begin{theorem}\label{groupconf}
In a model of an $\omega$-stable theory, each group configuration comes from an interpretable transitive action of an interpretable group on an interpretable set.
\end{theorem}
In this subsection the terms ``dimension'' and ``Morley rank'' means the same. Note that if there is a definable field in an $\omega$-stable structure, then we have a definable transitive action of a two-dimensional group on a one-dimensional set (the action by the group of affine transformations on the line). The following theorem of Hrushovski (see \cite[Theorem 3.27]{Po1}) is an inverse statement.
\begin{theorem}\label{gpactionfield}
In a model of an $\omega$-stable theory, if there is a definable transitive action of a two-dimensional interpretable group on a one-dimensional interpretable set, then there is an interpretable field.
\end{theorem}
By \emph{field configuration} we mean a group configuration coming from an action as above, i.e. a group configuration $\{g_1,g_2,g_3,x_1,x_2,x_3\}$ such that each $g_i$ has Morley rank two and each $x_j$ has Morley rank one. Using Theorem \ref{groupconf} and Theorem \ref{gpactionfield}, we immediately get the following.
\begin{theorem}\label{fieldconf}
If there is a field configuration in an $\omega$-stable structure $M$, then $M$ interprets a field.
\end{theorem}

\subsection{Interpretability of a field}\label{secex}
In this subsection, we find a $\mathcal{K}$-definable field. The proof follows the lines of the proof of \cite[Thm 7.3]{HaKo}.
\begin{lemma}\label{infiniteimage}
The image of the function
$$X\ni a\mapsto X'(a)\in K$$
is an infinite subset of $K$.
\end{lemma}
\begin{proof}
Let $X=X_1\cup \ldots \cup X_r$ be as in Lemma \ref{decomp1} and $V_i$ be the Zariski closure of $X_i$. If the image of the derivative function is finite, then on an infinite subset of (e.g.) $V_1$, the derivative function takes finitely many values. This happens only when $V_1$ is an affine line, contradicting Lemma \ref{needed}.
\end{proof}
\begin{theorem}\label{thereisfield}
There is a $\mathcal{K}$-interpretable field.
\end{theorem}
\begin{proof}
As in the proof of Proposition \ref{reduction}, we obtain in the proof below necessary ``$\mathbf{K}$-generic'' tuples just using the facts that $\acl_{\mathbf{K}}$ is a pregeometry and that $K$ is uncountable (see Assumption \ref{firstass}).

By Lemma \ref{deffunc}, the image from Lemma \ref{infiniteimage} is $\mathbf{K}$-definable and infinite. By Theorem \ref{swiss}, this image has a non-empty interior. Let $t\in K$ and $r\in |K^*|$ be such that the ball $B_t(r)$ is contained in the image of the function $a\mapsto X'(a)$. Without loss of generality, we may assume that $t=X'(0)$. Since the addition and the multiplication on $K$ are continuous in the valuation topology, there are $g_1,h_1\in B_1(r);\ g_2,h_2,b\in B_0(r)$ such that $\dim_{\mathbf{K}}(g_1,h_1,g_2,h_2,b)=5$ and
$$hg\in B_1(r)\times B_0(r);\ \  g\cdot b,hg\cdot b\in B_0(r),$$
where $g:=(g_1,g_2)\in \gm(K)\ltimes \ga(K)$ (similarly for $h$) and $\gm(K)\ltimes \ga(K)$ acts on $K$ by affine transformation: $g\cdot t = g_1t + g_2$. Then $\mathcal{G}:=(g,h,hg,b,g\cdot b, hg\cdot b)$ is a field configuration (see Section \ref{secfieldcon}) in the structure $(K,+,\cdot)$.

There is an open set $U\ni 0$ such that $(U\times U)\cap X$ is the graph of an analytic function (see Lemma \ref{deffunc}) and the ball $B_t(r)$ is contained in the image of this function. We will regard $X'$ as a function from $U$ to $K$.
We consider $6$-tuples
$$\mathcal{G}_{\mathcal{K}}:=(\alpha,\beta,\gamma,a,b,c),$$
where
$$\alpha=(\alpha_1,\alpha_2),\beta=(\beta_1,\beta_2),\gamma=(\gamma_1,\gamma_2);\ \ \ \alpha_1,\beta_1,\gamma_1,\alpha_2,\beta_2,\gamma_2,a,b,c\in U.$$
We can apply the function $X'$ coordinate-wise to such $6$-tuples. Since the ball $B_t(r)$ is contained in $X'(U)$, we obtain the following.
\\
\textbf{Claim 1} \emph{There is a a $6$-tuple $\mathcal{G}_{\mathcal{K}}$ as above such that the following holds for some 4-tuple $(g_1, g_2, h_1,h_2)$ in $K$.}
\begin{equation}
X'(\alpha_1)=tg_1,\  X'(\alpha_2)=t+g_2,\ X'(\beta_1)=th_1,\  X'(\beta_2)=t+h_2\tag{i}
\end{equation}
\begin{equation}
X'(\gamma_1)=th_1g_1,\  X'(\gamma_2)=t+h_1g_2+h_2\tag{ii}
\end{equation}
\begin{equation}
X'(a)=t+b,\ X'(b)=t+g_1b+g_2,\ X'(c)=t+h_1g_1b+h_1g_2+h_2\tag{iii}
\end{equation}
By Theorem \ref{fieldconf}, it is enough to show that $\mathcal{G}_{\mathcal{K}}$ from Claim 1 is a field configuration in the structure $\mathcal{K}$. The $\mathcal{K}$-Morley-rank conditions and the $\acl_{\mathcal{K}}$-independence conditions (see \cite[Def. 7.1]{HaKo}) follow easily, since the $\acl_{\mathbf{K}}$-independence is stronger than the $\acl_{\mathcal{K}}$-independence. We need to check the $\acl_{\mathcal{K}}$-dependence conditions. We will just check one of them, i.e. we will show that $b\in \acl_{\mathcal{K}}(\alpha,a)$.

By the conditions (i) -- (iii) from Claim 1, we compute the following.
\begin{IEEEeqnarray*}{rCl}
X'(b) & = & t+t^{-1}X'(\alpha_1)(X'(a)-t)+X'(\alpha_2)-t  \\
 & = & t^{-1}X'(\alpha_1)X'(a)-X'(\alpha_1)+X'(\alpha_2).
\end{IEEEeqnarray*}
Multiplying both sides by $t=X'(0)$ we get:
$$X'(0)X'(b)=X'(\alpha_1)X'(a)-X'(0)X'(\alpha_1)+X'(0)X'(\alpha_2).$$
Let us consider a $\mathcal{K}$-definable family $\left(Z_{(\delta)}\right)_{\delta\in X}$ where
$$Z_{(\delta)}:=X_{\alpha_1}\circ X_a-X_{\alpha_1}\circ X+X_{\alpha_2}\circ X-X_{\delta}\circ X.$$
By Lemma \ref{leibnitz}, we get that $Z_{(b)}'(0)=0$. For $\delta\in X$ let us denote
$$Z_{(\delta)}^{-1}(0):=Z_{(\delta)}\cap \{(x,y)\in K^2\ |\ y=0\}.$$
\textbf{Claim 2} \emph{There are infinitely many $\delta\in X$ such that}
$$\left|Z_{(\delta)}^{-1}(0)\right|>\left|Z_{(b)}^{-1}(0)\right|.$$
\begin{proof}[Proof of Claim 2]
Let us denote $Z_{(b)}$ by $V$. Using Lemma \ref{leibnitz}, the set $V$ satisfies the Assumption \ref{assume} and $V_0=\emptyset$ (the subscript ``$0$'' in ``$V_0$'' here is in the sense of Assumption \ref{assume}). Let
$$Z_{(b)}^{-1}(0)=\{a_1,\ldots,a_N\},$$
where $a_1=(0,0)$.
It is enough to show that there are pairwise disjoint open neighborhoods $U_i\ni a_i$ such that for infinitely many $\delta \in X$ we have:
$$|U_1\cap Z_{(\delta)}^{-1}(0)|\geqslant 2,\ |U_2\cap Z_{(\delta)}^{-1}(0)|\geqslant 1,\ldots,\ |U_N\cap Z_{(\delta)}^{-1}(0)|\geqslant 1.$$
Following \cite[Lemma 2.12]{MaPi}, we take an open neighborhood $U\ni (0,0)$ such that $U\cap V$ is the graph of a $\mathbf{K}$-analytic function $\varphi$. Since the set $V$ is of the form $W-X_{\delta}\circ X$ for a set $W$ containing $(0,0)$ and ``coming from $X$'' as in Remark \ref{comingfromx}, the function $\varphi$ is of the form
$$\varphi(x)=f(x)-g(h(x)),$$
where $f$ (respectively $g$ and $h$) is the $\mathbf{K}$-analytic function whose graph describes the set $W$ (respectively $X_\delta$ and $X$) in a neighborhood of $(0,0)$ (note that $(0,0)\in X_\delta$).

We define the following function
$$F(x,y)= f(x)-g(y-b_1+h(x)),$$
so that the graph of the function $F(\cdot,\gamma_1)$ describes the set $Z_{(\delta)}$ in a neighborhood of $(0,0)$. Let $U_1$ be the set $U_{a_1}\times U_b$, where $U_{a_1}$ and $U_b$ are given by Theorem \ref{ap}. Since $\frac{\partial F}{\partial x}(0,0)=0$, the function $F(\cdot , 0)$ has a multiple zero at $0$ (since $\ch(K)=0$). By Lemma 3.6, the set of $y$ for which $\frac{\partial F}{\partial x}(0,y)=0$ is finite and therefore, without loss of generality (i.e. possibly after shrinking $U_b$), we can suppose that for all $y\in (b_1+U_b) \setminus \{b_1\}$, $0$ is a simple zero of $F(\cdot , y)$. By Theorem \ref{ap}, for all $\gamma=(\gamma_1,\gamma_2)\in X\cap (b+U_1)$, we have
$$|U_1\cap Z_{(\delta)}^{-1}(0)|=|\{x\in U_1 \ | \ F(x,\gamma_1)=0\}|\geqslant 2.$$
Using Theorem \ref{ap}, the remaining sets $U_2,\ldots,U_N$ are obtained in a similar way.
\end{proof}
Let $N:=|Z_{(b)}^{-1}(0)|$. By Claim 2 and strong minimality of $X$, the $\{\alpha_1,\alpha_2,a\}$-definable (in the structure $\mathcal{K}$) set
$$D:=\left\{\delta\in X\ |\ N<|Z_{(\delta)}^{-1}(0)|\right\}$$
is cofinite. Since $b\in X\setminus D$, we obtain that $b\in \acl_{\mathcal{K}}(\alpha_1,\alpha_2,a)$ which finishes the proof.
\end{proof}
\begin{remark}
We would like to comment here on the characteristic $0$ assumption. Its most serious usage comes at the end of the proof of Claim above (zero derivative implies a multiple zero). This assumption was also used at the end of the proof of Lemma \ref{needed}(1).
\end{remark}

\subsection{Bi-interpretability with a field}

In this subsection, we show that the structure $\mathcal{K}$ is bi-interpretable with $(K,+,\cdot)$. Firstly, we need a lemma about $\mathbf{K}$-definable functions from $K$ to $K$.
\begin{lemma}\label{ktok}
There is no $\mathbf{K}$-definable function $f:K\to K$ such that the image of $f$ is infinite and all the fibers of $f$ are infinite.
\end{lemma}
\begin{proof}
By the paragraph above \cite[Problem 6.4]{HM1}, $\acl_{\mathbf{K}}$ is a pregeometry. Suppose that there is a $\mathbf{K}$-definable function $f:K\to K$ such that the image of $f$ is infinite and all the fibers of $f$ are infinite. Since $K$ is uncountable, there are $a\in f(K)\setminus \acl_{\mathbf{K}}(\emptyset)$ and $b\in f^{-1}(a) \setminus \acl_{\mathbf{K}}(a)$. But then $a\in \acl_{\mathbf{K}}(b)\setminus \acl_{\mathbf{K}}(\emptyset)$ and $b\notin \acl_{\mathbf{K}}(a)$ contradicting the Steinitz exchange principle.
\end{proof}
We can give now a description of $\mathbf{K}$-definable additive maps.
\begin{prop}\label{defend}
Any $\mathbf{K}$-definable endomorphism of $(K,+)$ is a scalar multiplication.
\end{prop}
\begin{proof}
Let $\phi$ be a $\mathbf{K}$-definable endomorphism of $(K,+)$. Arguing as in the proof of Lemma \ref{anfunc}, we see that $\phi$ is analytic when restricted to some non-empty open set. Since $\phi$ is additive, it is (in particular) analytic at $0$, as for each $a\in K$, we have $\phi(x)=\phi(a+x)-\phi(a)$. Let $F\in K\llbracket x\rrbracket$ be the Maclaurin series of $\phi$. Then $F$ is an additive formal power series. Since $\ch(K)=0$, there is $\lambda\in K$ such that $F=\lambda x$. Hence $\phi$ and the scalar multiplication by $\lambda$ coincide on an open neighborhood of $0$. Then the map
$$f:K\to K,\ \ f(x):=\lambda x-\phi(x)$$
has infinite fibers, so, by Lemma \ref{ktok}, $f$ has finite image. Therefore $f(K)$ is a finite subgroup of $(K,+)$, hence $f(K)=\{0\}$ and $\phi$ coincides with the the scalar multiplication by $\lambda$ everywhere on $K$.
\end{proof}
We will need valued field versions of the following o-minimal results: \cite[Thm. 1.1]{OtPePi} and \cite[Thm. 1.3]{PeSt}. Luckily for us, the first one is exactly Proposition 6.29 in \cite{HrMet}.
\begin{theorem}[Hrushovski]\label{otpepi}
Any $\mathbf{K}$-interpretable field is $\mathbf{K}$-definably isomorphic either to $(K,+,\cdot)$ or to the residue field.
\end{theorem}
The second one is rather easy to prove in the valued field context (using another theorem of Hrushovski).
\begin{theorem}\label{pest}
Any $\mathbf{K}$-definable, strongly minimal expansion of $(K,+,\cdot)$ coincides with $(K,+,\cdot)$.
\end{theorem}
\begin{proof}
Let $\mathcal{K}_1:=(K,+,\cdot)$ and let $\mathcal{K}_2$ be a $\mathbf{K}$-definable strongly minimal expansion of $\mathcal{K}_1$. By Theorem 1 of Section 3 in \cite{HR2}, it is enough to show that $\acl_{\mathcal{K}_1}=\acl_{\mathcal{K}_2}$.

For any $A\subseteq K$, we clearly have
$$\acl_{\mathcal{K}_1}(A)\subseteq \acl_{\mathcal{K}_2}(A)\subseteq \acl_{\mathbf{K}}(A).$$
Since the theory ACVF is model-complete (see Theorem \ref{qe}), $\acl_{\mathcal{K}_1}=\acl_{\mathbf{K}}$ (see also the second paragraph on p. 159 in \cite{HM1}). Hence we get $\acl_{\mathcal{K}_1}=\acl_{\mathcal{K}_2}$.
\end{proof}
We can prove now the main theorem of this paper. We recall that the structure $\mathbf{K}=(K,+,\cdot,\mathcal{O}_K)$ is a model of $\acvf_0$ and that $\mathcal{K}=(K,+,\cdots)$ is a strongly minimal $\mathbf{K}$-definable structure which is not locally modular.
\begin{theorem}\label{mainthm}
The structure $\mathcal{K}$ is bi-interpretable with the field $(K,+,\cdot)$.
\end{theorem}
\begin{proof}
Let $\mathcal{F}=(F,+_F,\cdot_F)$ be a $\mathcal{K}$-interpretable field given by Theorem \ref{thereisfield}. Our proof follows the lines of the proof of \cite[Thm. 7.4]{HaKo}. We divide the proof into three steps. \smallskip

{\bf Step 1.}
\emph{$(K,+)$ is definable in the structure $\mathcal{F}$}.
\\
Inside the structure $\mathcal{K}$, $K$ is non-orthogonal to $F$. By \cite[Cor. 2.27]{Po1}, $K$ is $F$-internal (still inside the structure $\mathcal{K}$). In particular, $\mathcal{F}$ is not $\mathbf{K}$-definably isomorphic to the residue field, since $K$ is not internal to the residue field inside the structure $\mathbf{K}$ (e.g. by \cite[Lemma 2.6.2]{HHM1}).
\\
Let $\mathcal{F}_{\mathcal{K}}$ be the structure $\mathcal{F}$ together with all the structure induced from $\mathcal{K}$. Then the structure $(K,+)$ is $\mathcal{F}_{\mathcal{K}}$-interpretable.
By Theorem \ref{otpepi} and the fact that $\mathcal{F}$ is not $\mathbf{K}$-definably isomorphic to the residue field, there is a $\mathbf{K}$-definable isomorphism
$$\Phi:\mathcal{F}\to (K,+,\cdot).$$
Hence the ``transported'' structure $\Phi(\mathcal{F})$ is a strongly minimal expansion of $(K,+,\cdot)$ which is definable in the structure $\mathbf{K}$. By Theorem \ref{pest}, the structure $\Phi(\mathcal{F})$ coincides with the structure $(K,+,\cdot)$, which implies that $\mathcal{F}=\mathcal{F}_{\mathcal{K}}$ (as structures). In other words, $\mathcal{K}$ does not induce any new structure on $\mathcal{F}$. Therefore, the structure $(K,+)$ is interpretable in the structure $\mathcal{F}$, which concludes Step 1. \smallskip

{\bf Step 2.}
\emph{There is an $\mathcal{F}$-definable isomorphism $\varphi:(K,+)\to (F,+_F)$}.
\\
By \cite[Thm. 4.13]{Po1} and Step 1, there is an $\mathcal{F}$-algebraic group $G$ and an  $\mathcal{F}$-definable isomorphism between $G(\mathcal{F})$ (the group of $\mathcal{F}$-rational points of $G$) and $(K,+)$. Since $\rmm_{\mathcal{K}}(K)=1$, we get $\rmm_{\mathcal{F}}(K)=1$ and $\dim(G)=1$. By the classification of connected one-dimensional algebraic groups, $G$ is isomorphic (as an $\mathcal{F}$-algebraic group) to $\ga$ or to $\gm$ or to an elliptic curve. Since $\ch(\mathcal{F})=0$ ($\mathcal{F}\cong (K,+,\cdot)$), $G$ is torsion free, and therefore $G\cong \ga$. In particular, there is an $\mathcal{F}$-definable isomorphism $(K,+)\cong (F,+_F)$.\smallskip

{\bf Step 3.}
\emph{There is $\lambda\in K$ such that for all $a,b\in K$ we have}
$$\varphi^{-1}\left(\varphi(a)\cdot_F\varphi(b)\right)=\lambda ab.$$
Let us define
$$*:K^2\to K,\ \ \ a*b=\varphi^{-1}(\varphi(a)\cdot_F\varphi(b)).$$
By Theorem \ref{otpepi} again, there is a $\mathbf{K}$-definable isomorphism
$$\psi:(K,+,*)\to (K,+,\cdot).$$
By Prop. \ref{defend}, there is $\lambda\in K^*$ such that $\psi$ is the scalar multiplication by $\lambda$, which finishes Step 3. \smallskip

By Step 3, we now know that the structure $(K,+,\cdot)$ is $(K,+,*)$-definable, so it is also $\mathcal{K}$-definable. Hence $\mathcal{K}$ is a $\mathbf{K}$-definable, strongly minimal expansion of $(K,+,\cdot)$. By Theorem \ref{pest} again, we obtain $\mathcal{K}=(K,+,\cdot)$ (as structures) which finishes the proof.
\end{proof}

\section{$C$-minimality and beyond}\label{seccmin}
Peterzil's conjecture is stated in the case of arbitrary o-minimal structures. A natural and important case is when the o-minimal structure is an expansion of an ordered field. In the valued field context, a natural replacement of the notion of o-minimality is the notion of \emph{$C$-minimality} (see \cite{HM1}). By \cite[Theorem C]{HM1}, $C$-minimal fields (with a possible extra structure) coincide with algebraically closed valued fields (with a possible extra structure) where the predicate $C$ comes from the valuation in the following way:
$$C(x;y,z)\ \ \ \ \ \text{if and only if}\ \ \ \ \  |x-y|>|y-z|.$$
A natural example of a $C$-minimal field being a proper expansion of a model of ACVF is the expansion of an algebraically closed valued field by rigid analytic spaces, which was considered by Lipshitz and Robinson \cite{Lipsh}.

Therefore, it is natural to ask whether Theorem \ref{mainthm} holds when ``algebraically closed valued field'' is replaced with ``$C$-minimal field''. One could also ask a similar question for $C$-minimal structures in general, obtaining a \emph{$C$-minimal version} of Peterzil's conjecture. However, there are examples of $C$-minimal structures where acl is \emph{not} a pregeometry (see \cite[page 121]{HM1}), so it is not clear whether $C$-minimality provides the necessary ``geometric flavor'' in general. On the other hand, it is still open (as far as we know) whether $\acl_{\mathbf{K}}$ is a pregeometry for a $C$-minimal field $\mathbf{K}$ (see \cite[Problem 6.4]{HM1}). Therefore, a version of Peterzil's conjecture in this case looks reasonable.

We will sketch below a possible attempt to prove a generalization of Theorem \ref{mainthm} to the case of a $C$-minimal field $\mathbf{K}$. The referee has pointed out several problems related to this approach, so it looks less optimistic than we had originally hoped. However, it is likely (as the referee pointed out as well) that this approach works for the expansion of an algebraically closed valued field by rigid analytic spaces.
\begin{enumerate}
\item A good decomposition of $X$ (Lemma \ref{decomp1}) should still hold. The quantifier elimination argument should be replaced by a $C$-minimal cell decomposition argument (see \cite{HM1} and \cite{Kov}).

\item Local analycity of the definable functions may follow from the decomposition above and Implicit Function Theorem (Theorem \ref{ift}). However, it is even unknown whether a definable function $K\to K$ is almost everywhere differentiable, see \cite{Del2}.

\item Once a version of Lemma \ref{needed} is be proved, the existence of a definable field (Section \ref{secex}) follows.

\item The full bi-interpretability theorem (Theorem \ref{mainthm}) formally follows from the following three items.
\begin{enumerate}
\item A description of the additive definable maps (Prop. \ref{defend}), which may follow from the $C$-minimal cell decomposition.

\item A description of $\mathbf{K}$-definable fields. This may be the most difficult part, since it is likely that Theorem \ref{otpepi} does \emph{not} hold for $C$-minimal fields in general.

\item A description of strongly minimal expansions of $(K,+,\cdot)$ (Theorem \ref{pest}). The following idea may work.
\\
By Lemma 1 of Section 3 of \cite{HR2}, we can assume that our expansion is of the form $(K,+,\cdot,f)$, where $f$ is a unary function. Arguing as in Section \ref{seccalc}, we see that $f$ is locally analytic. One can try to show that $f$ is $(K,+,\cdot)$-constructible by repeating the ``intersecting with lines'' argument from \cite{PeSt}.
\end{enumerate}
\end{enumerate}
We finish the paper with some remarks concerning the other cases of interest.
\begin{remark}
\begin{enumerate}
\item One can consider the case of a $p$-adically closed field instead of an algebraically closed valued field. A possible conjecture here may be formulated as follows: each strongly minimal structure definable in a $p$-adically closed field is locally modular.

\item One could also try to go beyond the $C$-minimality assumption in the case of algebraically closed valued fields. For example, it may be interesting to consider the Denef-Pas language (see \cite{pas}). In this case, Proposition \ref{defend} need not hold,
     so it is natural to expect that Theorem \ref{mainthm} holds ``up to (the action of) $\aut_{\mathbf{K}}(K,+)$''. Note that this is exactly the case in \cite{HaKo}, where a bi-interpretability result is obtained up to the action of $\gl_2(\Rr)$, and $\gl_2(\Rr)$ coincides with $\aut_{\mathcal{R}}(\Cc,+)$.
\end{enumerate}
\end{remark}

\bibliographystyle{plain}
\bibliography{harvard}

\end{document}